\documentclass[reqno,12pt]{amsart}
\usepackage{amsfonts}
\usepackage{bbm}
\usepackage{}
\numberwithin{equation}{section}
\usepackage{color}
\usepackage{xcolor}
\usepackage{cite}
\usepackage{mathtools}
\usepackage{indentfirst}
\usepackage{color}
\usepackage{amssymb}
\usepackage{mathrsfs}
\usepackage{xy}
\xyoption{all}

\theoremstyle{plain}
\newtheorem{theorem}{\bf Theorem}[section]
\newtheorem{lemma}[theorem]{\bf Lemma}
\newtheorem{corollary}[theorem]{\bf Corollary}
\newtheorem{proposition}[theorem]{\bf Proposition}

\theoremstyle{definition}
\newtheorem{definition}[theorem]{\bf Definition}
\newtheorem{remark}[theorem]{\bf Remark}
\newtheorem{example}[theorem]{\bf Example}

\newtheorem{algorithm}[theorem]{\bf Algorithm}

\newcommand{\bt}{\begin{theorem}}
	\newcommand{\et}{\end{theorem}}
\newcommand{\bl}{\begin{lemma}}
	\newcommand{\el}{\end{lemma}}
\newcommand{\bd}{\begin{definition}}
	\newcommand{\ed}{\end{definition}}
\newcommand{\bc}{\begin{corollary}}
	\newcommand{\ec}{\end{corollary}}
\newcommand{\bp}{\begin{proof}}
	\newcommand{\ep}{\end{proof}}
\newcommand{\bx}{\begin{example}}
	\newcommand{\ex}{\end{example}}
\newcommand{\br}{\begin{remark}}
	\newcommand{\er}{\end{remark}}
\newcommand{\be}{\begin{equation}}
	\newcommand{\ee}{\end{equation}}
\newcommand{\ba}{\begin{align}}
	\newcommand{\ea}{\end{align}}
\newcommand{\bn}{\begin{enumerate}}
	\newcommand{\en}{\end{enumerate}}
\newcommand{\bcs}{\begin{cases}}
	\newcommand{\ecs}{\end{cases}}

\makeatletter
\renewcommand{\section}{\@startsection{section}{1}{0mm}
	{-\baselineskip}{0.5\baselineskip}{\bf\leftline}}
\makeatother

\begin{document}
	
	\title[A similarity canonical form for max-plus matrices and its eigenproblem]{A similarity canonical form for max-plus matrices and its eigenproblem}
		\author{Haicheng Zhang and Xiyan Zhu}
	
\address{Ministry of Education Key Laboratory of NSLSCS, School of Mathematical Sciences, Nanjing Normal University,
Nanjing 210023, P.~R.~China}
\email{zhanghc@njnu.edu.cn (Zhang)}
\address{School of Mathematical Sciences, Nanjing Normal University, Nanjing 210023, P.~R.~China}
	\email{06210513@njnu.edu.cn (Zhu)}

	
	\begin{abstract}
We provide a necessary and sufficient condition for matrices in the max-plus algebra to be pseudo-diagonalizable, calculate the powers of pseudo-diagonal matrices and prove the invariance of optimal-node matrices and separable matrices under similarity. As an application, we determine the eigenvalues and eigenspaces of pseudo-diagonalizable matrices.
	\end{abstract}
	
	\maketitle
	
\section{Introduction}
The max-plus algebra is just the semi-ring $\overline{\mathbb{R}}=\{\mathbb{R}\cup\{-\infty\},\oplus,\otimes\}$ with operations defined by $a\oplus b=max\{a,b\}$ and $a\otimes  b=a+b$. The additive and multiplicative identities are $\varepsilon=-\infty$ and $0$, respectively. The max-plus algebra is one of idempotent semi-rings which have been applied in various fields, and it is becoming increasingly useful because it can transform certain nonlinear systems appearing for instance in scheduling theory, automata theory and discrete event system into linear ones, such that familiar linear algebra techniques are applicable (see e.g.\cite{cuninghame1, cuninghame2,gondran1,zimmermann}).

In the max-plus algebra, the similarity relation of matrices plays a crucial role, as it can simplify the form of a matrix while preserving most of its properties (see e.g. \cite{bapat1,brualdi1,butkovic1,butkovic2,rosen}). Since similarity is an equivalence relation, it's natural to find a canonical form for each such equivalence class. In classical linear algebra, significant efforts have been devoted to exploring this topic, and remarkable results have been achieved, for example, the real symmetric matrices are diagonalizable. However, this is not true in the max-plus algebra, where a diagonal matrix is defined as a matrix with real numbers on the diagonal and $\varepsilon$ elsewhere. It's easy to verify that a diagonal matrix can only be similar to another diagonal matrix with the same diagonal entries. Recently, Mukherjee \cite{arxiv} proposed a canonical form for the congruence relation of matrices in the max-plus algebra, called the pseudo-diagonal matrix. In this paper, we study which matrices are similar to pseudo-diagonal matrices.

The eigenproblem, closely related to similarity, is of key importance in the max-plus algebra. It has been studied from the 1960's in connection with the analysis of a steady-state behaviour of certain production systems (see e.g. \cite{cuninghame3}). A full solution in the case of irreducible matrices has been provided in \cite{cuninghame1} and a general spectral theorem was given in \cite{gaubert1}. For a general $n\times n$ real matrix, there exist $O(n^3)$ algorithms to find its maximum cycle mean. For some special kinds of matrices, the computation can be performed in simpler ways. Thus, from the point of view of computation, the discussion of eigenproblem for special matrices is significant (see e.g.\cite{butkovic3,gavalec1,gavalec2,plavka1,plavka2,wanghuili}). In this paper, we study the eigenproblem of the matrices which are similar to pseudo-diagonal matrices.
	
This paper is organized as follows. Firstly, some necessary notions and results in the max-plus algebra are given in Section $2$. We provide a necessary and sufficient condition for matrices in the max-plus algebra to be pseudo-diagonalizable, as well as an algorithm for verifying pseudo-diagonalizable matrices in Section 3. In Section 4, we calculate the powers of pseudo-diagonal matrices, prove the invariance of optimal-node matrices and separable matrices under
similarity, and determine optimal-node matrices and separable matrices in pseudo-diagonalizable
matrices. In the final section, as an application, we determine the eigenvalues and eigenspaces of pseudo-diagonalizable matrices.
\section{Preliminaries}
Recall that the max-plus algebra $\overline{\mathbb{R}}=\{\mathbb{R}\cup\{-\infty\},\oplus,\otimes\}$ has operations defined by $a\oplus b=max\{a,b\}$ and $a\otimes  b=a+b$, and the additive and multiplicative identities are $\varepsilon=-\infty$ and $0$, respectively.
Clearly, for any $a\in \overline{\mathbb{R}}$, we have $a\oplus\varepsilon=a$ and $a\otimes\varepsilon=\varepsilon$. For each positive integer $k$ and $a\in \overline{\mathbb{R}}$, we denote by $a^k$ the $k$-th power of $a$ with respective to the operation $\otimes$. For any $a,r\in\mathbb{R}$, $a^{r}$ denotes the usual multiplication $ra$ in $\mathbb{R}$.

For two matrices $A=(a_{ij})$ and $B=(b_{ij})$ in $\overline{\mathbb{R}}$, if $A$ and $B$ are both of $m\times n$ type, define $A\oplus B$ to be the matrix $C=(c_{ij})_{m\times n}$ with $c_{ij}=a_{ij}\oplus b_{ij}$ for all $i,j$; if $A$ is of $m\times s$ type and $B$ is of $s\times n$ type, define $A\otimes B$ to be the matrix $C=(c_{ij})_{m\times n}$ with $c_{ij}=\sum_{k=1}^{s}a_{ik}\otimes b_{kj}$ for all $i$,$j$, where $\sum$ is the summation with respective to the operation $\oplus$; and define $\lambda\otimes A$ to be the matrix $(\lambda\otimes a_{ij})_{m\times s}$ for any $\lambda\in\overline{\mathbb{R}}$. For two matrices $A=(a_{ij})$ and $B=(b_{ij})$ of the same type in $\overline{\mathbb{R}}$, define $A\le B$ if and only if $a_{ij}\le b_{ij}$ for all $i,j.$ If $A$ is square, for each positive integer $k$ we denote by $A^k$ the $k$-th power of $A$ with respective to the operation $\otimes$. We denote by $\overline{\mathbb{R}}^{m\times n}$ the set of all matrices of $m\times n$ type in $\overline{\mathbb{R}}$, similarly, we have the notation ${\mathbb{R}}^{m\times n}$. For any positive integer $n$, we use $[n]$ to represent the set $\{1,2,\dots,n\}$. Let $\mathcal {P}_n$ be the set of cyclic permutations of subsets of $[n]$. For simplicity of notation, we may also write $ab$ and $AB$ for $a\otimes b$ and $A\otimes B$, respectively.
We denote the vector whose entries are all $\varepsilon$ by $\boldsymbol{\varepsilon}$, and the usual zero vector and zero matrix are both denoted by $0$. For a finite set $X$, we denote by $|X|$ the cardinality of $X$.
	
Given a non-empty subset $S\subseteq\overline{\mathbb{R}}^{n}:=\overline{\mathbb{R}}^{n\times 1}$, we say that a vector $v\in\overline{\mathbb{R}}^n$ is a \emph{max-combination} of $S$ if $v=\sum_{x\in S}\alpha_x\otimes x$ for some $\alpha_x\in\overline{\mathbb{R}}$, where only a finite number of $\alpha_x$ are finite. The set $S$ is said to be \emph{dependent} if there exists an element $v\in S$ such that $v$ is a max-combination of $S\setminus\{v\}$. Otherwise, $S$ is \emph{independent}. The non-empty set $S$ is called a \emph{subspace} if $a\otimes u\oplus b\otimes v\in S$ for any $u,v\in S$ and $a,b\in\overline{\mathbb{R}}$.  Clearly, the set $\{\boldsymbol{\varepsilon}\}$ and $\overline{\mathbb{R}}$ are subspaces, called {\em trivial subspaces}. The set of all max-combinations of $S$ is denoted by span($S$). It is easy to see that $T:=$span($S$) is a subspace, and then $S$ is called a \emph{set of generators} of $T.$  The set $S$ is called a \emph{basis} of $T$ if it is an independent set of generators for $T$. Every basis of a finitely generated subspace $T$ has the same number of elements (c.f.\cite{butkovic2}), and the number of elements in any basis is called the \emph{dimension} of the subspace $T$.
For any matrix $A\in\overline{\mathbb{R}}^{m\times n}$, we always denote by $A_i$ the $i$-th column vector of $A$ for any $i\in[n]$.

Given a matrix $A\in\overline{\mathbb{R}}^{n\times n}$, the problem of finding vectors $x\in\overline{\mathbb{R}}^n$ with $x\neq\boldsymbol{\varepsilon}$ and $\lambda\in\overline{\mathbb{R}}$ such that
	\begin{equation*}
		A\otimes x=\lambda\otimes x
	\end{equation*}
	is called the \emph{eigenproblem} in the max-plus algebra, and such $x$ and $\lambda$ are called the \emph{eigenvectors} and \emph{eigenvalues} of $A$, respectively. The \emph{eigenspace} of $A$, denoted by $V(A)$, is the set consisting of $\boldsymbol{\varepsilon}$ and all eigenvectors of $A$. The dimension of $V(A)$ is denoted by $d(A)$. For any $\sigma=(i_1,i_2,\dots,i_k)\in\mathcal {P}_n$, $k$ is called the length of $\sigma$, denoted by $l(\sigma)$. We define the weight of $\sigma$ corresponding to $A$ as
	\begin{equation*}
		\omega(\sigma,A)=a_{i_1i_2}a_{i_2i_3}\dots a_{i_{k-1}i_k}a_{i_ki_1}.
	\end{equation*}
	Furthermore, the \emph{mean} of $\sigma$ is defined as
	\begin{equation*}
		\mu(\sigma,A)=\frac{\omega(\sigma,A)}{l(\sigma)},
	\end{equation*}
	and the \emph{maximum cycle mean} of $A$ is defined as
	\begin{equation*}
		\lambda(A)=\mathop{max}\limits_{\sigma\in\mathcal {P}_n}\mu(\sigma,A).
	\end{equation*}
\begin{theorem}\label{lamdaeigen}\textup{(\hspace{-0.005cm}\cite{bapat2})}
Let $A\in\overline{\mathbb{R}}^{n\times n},$ then $\lambda(A)$ is the greatest eigenvalue of $A$. If $A\in\mathbb{R}^{n\times n}$, then $\lambda(A)$ is the unique eigenvalue of $A$ and $V(A)\setminus\{\boldsymbol{\varepsilon}\}\subseteq\mathbb{R}^n$.
\end{theorem}

Given $\sigma\in\mathcal {P}_n$, it is called a \emph{critical cycle} if $\mu(\sigma,A)=\lambda(A)$. The nodes appearing in critical cycles are called the \emph{critical nodes} of $A$. The set of critical nodes of $A$ is denoted by $N_c(A).$
	We say $i,j\in N_c(A)$ are \emph{equivalent}, denoted by $i\sim j$, if they belong to the same critical cycle, otherwise, they are called \emph{non-equivalent} and denoted by $i\nsim j.$ Clearly, this gives an equivalence relation on $N_c(A),$ and for any $i\in N_c(A)$ we denote by $\bar{i}$ the equivalence class of $i$.	
For any $A\in\mathbb{R}^{n\times n}$, define the \emph{transitive closure} of $A$ as $\Gamma(A)=A\oplus A^2\oplus\dots\oplus A^n.$
	Let $A_\lambda=(\lambda(A))^{-1}\otimes A$. Then we have that $\Gamma(A_\lambda)$ contains at least one column in which the diagonal entry of $\Gamma(A_\lambda)$ is $0$ and all such columns are eigenvectors of $A$ (c.f.\cite{cuninghame1}), which are called \emph{fundamental eigenvectors}. For each $i\in[n]$, we write $g_i$ for the $i$-th column vector $\Gamma(A_\lambda)_i$ of $\Gamma(A_\lambda)$.
	\begin{theorem}\textup{(\hspace{-0.005cm}\cite{akian})}\label{eigenspace2.4}
		For any $A\in\mathbb{R}^{n\times n}$, $V(A)$ is a nontrivial subspace and we have a basis of $V(A)$ by taking exactly one $g_{\bar{i}}$ for each equivalence class $\bar{i}$ in $N_c(A)/\sim.$
	\end{theorem}
A matrix $A$ is called \emph{finite}, if all entries of $A$ are taken from $\mathbb{R}$.
A square matrix is called \emph{diagonal}, if its diagonal entries are taken from $\mathbb{R}$ and off-diagonal entries are $\varepsilon$. Denote by ${\rm diag}(d_1,\dots,d_n)$ the diagonal matrix with diagonal entries $d_1,\dots,d_n\in\mathbb{R}$. The matrix $I:={\rm diag}(0,\dots,0)$ is called the {\em unit matrix}, and for a finite square matrix $A$, we may write $A^0$ for $I$. Any matrix which can be obtained from the unit (resp. diagonal) matrix by permuting the rows or columns is called a \emph{permutation matrix} (resp. \emph{generalized permutation matrix}). Obviously, given a generalized permutation matrix $A=(a_{ij})\in\overline{\mathbb{R}}^{n\times n}$, there is a unique permutation $\pi$ of $[n]$, called the \emph{permutation of} $A$, such that for any $i, j\in [n]$ we have:
	\begin{equation}\label{gpyx}
		a_{ij}\in\mathbb{R}\Longleftrightarrow j=\pi(i).
	\end{equation}
	For each matrix $A\in\overline{\mathbb{R}}^{n\times n}$, it is called {\em invertible}, if there exists a matrix $B\in\overline{\mathbb{R}}^{n\times n}$ such that $A\otimes B=I=B\otimes A$. It is easy to see that if such $B$ exists, it is unique, called the {\em inverse} of $A$ and denoted by $A^{-1}$. Generalized permutation matrices play a crucial role in the max-algebra as they are the only invertible matrices.

In what follows, for convenience and brevity, we write $A_{ij}$ for the $(i,j)$-entry of any matrix $A$. In particular, $(A^k)_{ij}$ denotes the $(i,j)$-entry of the $k$-th power of any square matrix $A$, and $(A^{-1})_{ij}$ denotes the $(i,j)$-entry of the inverse of any invertible matrix $A$.

\begin{theorem}\label{kelip}
{\textup{(\hspace{-0.005cm}\cite{cuninghame1})}}
Let $A=(a_{ij})\in\overline{\mathbb{R}}^{n\times n}$. Then $A$ is invertible
if and only if $A$ is a generalized permutation matrix. Moreover,  if $A$ is invertible, the entries of $A^{-1}$ are given as follows:
\begin{flalign*}
(A^{-1})_{ij}=\begin{cases}
a_{ji}^{-1} &i=\pi(j)\\
\varepsilon &i\neq\pi(j),
\end{cases}
\end{flalign*}
where $\pi$ is the permutation of $A$.
\end{theorem}
	
	\begin{definition}
For any matrices $A,B\in\overline{\mathbb{R}}^{n\times n}$, we say that $A$ and $B$ are \emph{similar}, if there exists an invertible matrix $P$ such that
$B=P^{-1}AP$.
	\end{definition}
	Let $A,B,P\in\overline{\mathbb{R}}^{n\times n}$ and $P$ be a generalized permutation matrix with permutation $\pi$. Suppose that $A=PBP^{-1}$, then we have
\begin{equation}\label{xsyg}
		\begin{split}
	A_{ij}&=\sum_{k=1}^{n}\sum_{t=1}^{n}P_{ik}B_{kt}(P^{-1})_{tj}\\
	&=P_{i,\pi(i)}B_{\pi(i),\pi(j)}(P_{j,\pi(j)})^{-1}.
\end{split}
	\end{equation}
	In particular, if $P$ is also a permutation matrix, then $A_{ij}=B_{\pi(i),\pi(j)}$.
\begin{lemma}\label{similar eigen}
	Let $A,B\in\mathbb{R}^{n\times n}$ such that $B=PAP^{-1}$ for some invertible matrix $P$, then $\lambda(A)=\lambda(B)$ and $V(B)=\{P\otimes x:x\in V(A)\}$.
	\end{lemma}
	\begin{proof}
	Let $x\in\mathbb{R}^n$ be an eigenvector of $A$ satisfying $A\otimes x=\lambda(A)\otimes x$. Since $$B\otimes P\otimes x=P\otimes A\otimes x=\lambda(A)\otimes P\otimes x,$$ we obtain that $P\otimes x$ is an eigenvector of $B$ corresponding to the eigenvalue $\lambda(A)$. By Theorem \ref{lamdaeigen}, we have $\lambda(B)=\lambda(A)$. By the invertibility of $P$, we obtain that $V(B)=\{P\otimes x:x\in V(A)\}.$
	\end{proof}
	\section{Pseudo-diagonal matrices as canonical forms under similarity}
In this section, we determine which matrices in the max-plus algebra are similar to the so-called pseudo-diagonal matrices. Firstly, let us give the related definitions as follows.
\begin{definition}
A square matrix is called \emph{pseudo-diagonal}, if its diagonal entries are taken from $\mathbb{R}$ and off-diagonal entries are $0$.
A square matrix is called \emph{pseudo-diagonalizable}, if it is similar to a pseudo-diagonal matrix. Denote by ${\rm pdiag}(d_1,\dots,d_n)$ the pseudo-diagonal matrix with diagonal entries $d_1,\dots,d_n$.
\end{definition}

Let us provide an example of applications of pseudo-diagonal matrices in the practice.
	\begin{example}\label{example MMIPP}
		Suppose that in the \emph{multi-machine interactive production process}, abbreviated as MMIPP, there exist $n$ machines which work in stages. Assume that in each stage all machines simultaneously produce necessary components for the next stage of some or all other machines. For each $i\in[n]$ and non-negative integer $r$, let $x_i(r)$ denote the starting time of the $i$-th machine in the $r$-th stage, and let $a_{ij}$ denote the duration of the operation at which the $j$-th machine prepares necessary components for the $i$-th machine in the next stage. Then \begin{equation*}
			x_i(r + 1)=max(x_1(r)+a_{i1},\dots,x_n(r)+a_{in}),~~i\in [n], r\geq0.
		\end{equation*}
In max-algebraic notation, we have
\begin{equation*}
			x(r+1)=A\otimes x(r),~~r\geq0,
		\end{equation*}
		where $x(r)=(x_1(r),\ldots,x_n(r))^{\rm T}$ and $A=(a_{ij})$ is called a production matrix (c.f. \cite{butkovic2,cuninghame3}).
		
If the production matrix $A$ is pseudo-diagonal, by definition, $a_{ij}=0$ for any $i\neq j$. So, in this case, we have
\begin{equation*}
			x_i(r+1)=max(x_1(r),\dots,x_{i}(r)+a_{ii},\dots,x_n(r)),~~i\in [n],r\geq0,
		\end{equation*}
which means that the $i$-th machine starts the $(r+1)$-th stage if and only if all machines have been in the $r$-th stage and the $i$-th machine has finished preparing necessary components for the next stage.
	\end{example}

Next, we provide a necessary and sufficient condition for a matrix to be pseudo-diagonalizable.
	
	\begin{theorem}\label{theorem3.5}
		Let $A=(a_{ij})\in\overline{\mathbb{R}}^{n\times n}$, then $A$ is pseudo-diagonalizable if and only if $A$ is finite and the equation \begin{flalign}\label{3.8}
			a_{ik}\otimes a_{kj}=\begin{cases}
				a_{ij} &i\neq j,\\
				0 &i=j
			\end{cases}
		\end{flalign}
		holds for any $i,j,k\in [n]$ with $k\neq i,j.$
	\end{theorem}
	\begin{proof}
	Suppose that $A\in\overline{\mathbb{R}}^{n\times n}$ is a pseudo-diagonalizable matrix, by definition and Theorem \ref{kelip}, there exists a generalized permutation matrix $P$ with permutation $\pi$ such that $P^{-1}AP=D$, where $D={\rm pdiag}(d_1,\dots,d_n)$ for some $d_1,\ldots,d_n\in\mathbb{R}$. By \eqref{xsyg}, we obtain that $a_{ij}=P_{i,\pi(i)}D_{\pi(i),\pi(j)}(P_{j,\pi(j)})^{-1}$ for any $i,j\in[n]$.
		Set $p_i=P_{i,\pi(i)}$ for each $i\in[n]$, then
			\begin{flalign}\label{3.9}
			a_{ij}=\begin{cases}
				d_{\pi(i)} &i=j,\\
				p_ip_j^{-1} &i\neq j.
			\end{cases}
		\end{flalign}
Thus, $A$ is a finite matrix.
For any $k\in[n]$ with $k\neq i,j$, by the equation (\ref{3.9}), we obtain that
\begin{flalign*}
			a_{ik}a_{kj}=p_ip_k^{-1}p_kp_{j}^{-1}=p_ip_j^{-1}=\begin{cases}
				a_{ij} &i\neq j,\\
				0 &i=j.
			\end{cases}
		\end{flalign*}

		Conversely, suppose that $A=(a_{ij})\in\mathbb{R}^{n\times n}$ satisfies \eqref{3.8}. Let $(p_1,\ldots,p_n)^{\rm T}\in\mathbb{R}^n$ be a solution of the following system of linear equations in classical linear algebra
\begin{flalign*}
			\begin{cases}
				x_1-x_2=a_{21},\\
				x_2-x_3=a_{32},\\
				\dots \\
				x_{n-1}-x_n=a_{n,n-1}.\\
			\end{cases}
		\end{flalign*}
Let $P$ be the diagonal matrix ${\rm diag}(p_1,\ldots,p_n)$. By Theorem \ref{kelip}, $P$ is invertible. Set $D=PAP^{-1}$. According to \eqref{gpyx}, the permutation of $P$ is the identity permutation. Then by \eqref{xsyg}, for any $i,j\in[n]$ we have that
\begin{equation}\label{Dij}D_{ij}=p_ia_{ij}p_j^{-1}.\end{equation}
For $1\leq j<i\leq n$, by \eqref{3.8}, we have that
\begin{flalign*}
			a_{ij}&=a_{i,i-1}a_{i-1,j}\\
			&=\cdots\\
			&=a_{i,i-1}a_{i-1,i-2}\dots a_{j+1,j}\\
			&=p_i^{-1}p_{i-1}p_{i-1}^{-1}p_{i-2}\dots p_{j+1}^{-1}p_{j}\\
			&=p_i^{-1}p_j.
		\end{flalign*}
While for $1\leq i<j\leq n$, by \eqref{3.8}, we have that $a_{ij}=a_{ji}^{-1}=p_i^{-1}p_j$. Hence, by \eqref{Dij}, we obtain that $D_{ij}=0$ for any $1\leq i\neq j\leq n$.
Clearly, $D_{ii}=a_{ii}\in\mathbb{R}$. Therefore, $D$ is a pseudo-diagonal matrix.
	\end{proof}
	
	The following algorithm is used to verify whether a given matrix $A\in\mathbb{R}^{n\times n}$ is a pseudo-diagonalizable matrix.
	\begin{algorithm}\label{Algorithm3.4}\ \\
		Input: $A=(a_{ij})\in\mathbb{R}^{n\times n}.$\\
		Output: The indication of whether $A$ is a pseudo-diagonalizable matrix.\\
		Step $1$\\
		For all $i,j\in [n]$ with $i\textless j$ do
		
		begin
		
		$K_{ij}=a_{ij}a_{ji}$
		
		end\\
		If all $K_{ij}=0$, proceed to the next step. Otherwise, $A$ is not pseudo-diagonalizable.\\
		Step $2$\\
		For $i=1,2,\dots,n-2$ do
		
		for $j=i+2,i+3,\dots,n$ do
		
		begin
		
		$T_{ij}=a_{ij}a_{j,i+1}a_{i+1,i}$
		
		end\\
		If all $T_{ij}=0$, then $A$ is pseudo-diagonalizable. Otherwise, $A$ is not pseudo-diagonalizable.
	\end{algorithm}
	\begin{theorem}
		\rm{Algorithm} \rm{\ref{Algorithm3.4}} is correct and terminates after $O(n^2)$ operations.
	\end{theorem}
	\begin{proof}
		Correctness: if $A$ is pseudo-diagonalizable, first of all, by the second equation in (\ref{3.8}), for any $1\leq i\neq j\leq n$ we have that
$$K_{ij}=a_{ij}a_{ji}=0.$$
Then for any $i,j$ with $1\leq i\neq j\neq i+1\leq n$, by the first equation in (\ref{3.8}), we have that
\begin{equation*}T_{ij}=a_{ij}a_{j,i+1}a_{i+1,i}=a_{ij}a_{ji}=0.\end{equation*}
In particular, for any $i,j\in[n]$ with $i\le n-2$ and $j\ge i+2$, we have that $T_{ij}=0$.
		
Conversely, suppose that $K_{ij}=0$ for any $1\leq i<j\leq n$ and $T_{ij}=0$ for any $i,j\in[n]$ with $i\le n-2$ and $j\ge i+2$.

Firstly, for any $1\leq i\neq k\leq n$, if $i<k$, then $K_{ik}=0$, i.e. $a_{ik}a_{ki}=0$; if $i>k$, then $a_{ki}a_{ik}=0$.
Thus, we get the second equation in (\ref{3.8}).

Secondly, let $i,j\in [n]$ with $i\neq j$, we need to prove $a_{ij}=a_{ik}a_{kj}$ for any $k\in[n]$ with $i\neq k\neq j$.
Without loss of generality, we assume $i<j$.

If $i<k<j$, then $i\le n-2$ and $j\ge i+2$, so
$T_{ij}=a_{ij}a_{j,i+1}a_{i+1,i}=0$. Thus, we have that
\begin{equation}\label{3.12}
a_{ij}=(a_{j,i+1}a_{i+1,i})^{-1}=a_{i,i+1}a_{i+1,j}.
\end{equation}
That is, for $k=i+1$ the first equation in (\ref{3.8}) holds.

If $i+1<k<j$, then by \eqref{3.12} we have that $a_{i+1,j}=a_{i+1,i+2}a_{i+2,j}$. Thus, again by \eqref{3.12},
\begin{flalign*}
		a_{ij}&=a_{i,i+1}a_{i+1,j}\\\
		&=a_{i,i+1}a_{i+1,i+2}a_{i+2,j}\\
		&=a_{i,i+2}a_{i+2,j}.
	\end{flalign*}
It means that for $k=i+2$ the first equation in (\ref{3.8}) holds.

Repeating this process, we obtain that $a_{ik}a_{kj}=a_{ij}$ for any $k=i+1,\ldots,j-1$.
Hence, we have proved that $a_{ij}=a_{ik}a_{kj}$ for $i\textless k\textless j$.

If $k<i<j$, then $a_{kj}=a_{ki}a_{ij}$, and thus $a_{ij}=a_{ki}^{-1}a_{kj}=a_{ik}a_{kj}$. If $i<j<k$, then $a_{ik}=a_{ij}a_{jk}$, and thus $a_{ij}=a_{ik}a_{jk}^{-1}=a_{ik}a_{kj}$.

Therefore, we complete the proof of (\ref{3.8}), and thus $A$ is pseudo-diagonalizable.

Complexity bound: the complexity of Step $1$ is $O(\frac{n(n-1)}{2})=O(n^2)$, and the complexity of Step $2$ is $O(\sum_{i=1}^{n-2}i)=O(\frac{(n-1)(n-2)}{2})=O(n^2)$. Thus, the complexity bound is $O(n^2)$.
	\end{proof}

\section{Pseudo-diagonal matrix powers and some special matrices}
In this section, we calculate the powers of pseudo-diagonal matrices, recall the definitions and properties of optimal-node matrices and separable matrices, prove the invariance of optimal-node matrices and separable matrices under similarity, and finally determine
optimal-node matrices and separable matrices in pseudo-diagonalizable matrices.
\subsection{Powers of pseudo-diagonal matrices}
The computation of matrix powers is an important problem in the max-plus algebra, which is close related to the calculation of eigenvalues and eigenvectors, as well as the problem of reachability of eigenspaces (c.f. \cite{butkovic2,butkovic4}). Jones \cite{daniel} provided explicit formulas for powers of $2\times 2$ finite matrices. A useful and well-known result on matrix powers is the Cyclicity Theorem. For any matrix $A\in\mathbb{R}^{n\times n}$, define the \emph{cyclicity} of $A$ as the greatest common divisor of the lengths of the critical cycles of $A$, and define the {\em period} of $A$ as the smallest positive integer $p$ for which there is a positive integer $T$ such that for any integer $k\ge T,$
\begin{equation*}
		A^{k+p}=(\lambda(A))^p\otimes A^k.
	\end{equation*}
	
\begin{theorem}\textup{(Cyclicity Theorem, \cite{butkovic2})}
	For any matrix $A\in\mathbb{R}^{n\times n}$, the
	cyclicity of $A$ is equal to the period of $A$.
\end{theorem}

\begin{lemma}\label{4.17}
	Let $D={\rm pdiag}(d_1,\dots,d_n)$ be a pseudo-diagonal matrix and $d=max\{d_1,\dots,d_n\}$. Then we have the following:
	
	$(1)$ $\lambda(D)=d\oplus0$.
	
	$(2)$ If $d\textgreater0$, then $i$ is a critical node of $D\Longleftrightarrow d_i=d.$
	
	$(3)$ If $d\textgreater0$, then $N^*_c(D)=N_c(D).$
\end{lemma}
\begin{proof}
For any cycle $(i)$ of length one with $i\in[n]$, $\mu((i),D)=D_{ii}=d_i$.
While for any cycle $\sigma=(i_1,\dots,i_l)$ with $l\ge2$, we have that
	\begin{equation*}
		\omega(\sigma,D)= {D_{i_1i_2}+D_{i_2i_3}+\dots +D_{i_{l-1}i_l}+D_{i_li_1}}=0,	
	\end{equation*}
and then $\mu(\sigma,D)=0$. Thus, $\lambda(D)=d\oplus0.$ This finishes the proof of $(1)$.
	
If $d>0$, $\lambda(D)=d\oplus0=d.$ Since for any cycle $\sigma$ with $l(\sigma)\geq 2$, $\mu(\sigma,D)=0\neq d$, the critical nodes only appear in the cycles of length one. So for each $i\in [n]$, $i$ is a critical node of $D$ if and only if $d_i=d$, and for any critical nodes $i,j$, we have that $i\sim j\Longleftrightarrow i=j$. Hence, we finish the proofs of $(2)$ and $(3)$.
\end{proof}
\begin{lemma}\label{corollary cyclicity}
	Let $D={\rm pdiag}(d_1,\dots,d_n)$ be a pseudo-diagonal matrix. If $n=2$ and $d_1\oplus d_2<0$, then the cyclicity of $D$ is $2$. Otherwise, the cyclicity of $D$ is $1$.
\end{lemma}
\begin{proof}
It is easily proved by definitions and Lemma \ref{4.17}.
\end{proof}

In order to illustrate the applications of matrix powers in the problem of reachability of eigenspaces, let us recall the system in Example \ref{example MMIPP}
\begin{equation}\label{xitong}
	x(r+1)=A\otimes x(r),~~ r\geq 0.
\end{equation}
We say that the system \eqref{xitong} reaches a \emph{steady state regime} if it eventually moves forward in regular steps, that is, there exists some integer $k\geq0$ such that $$A^{k+1}\otimes x(0)=\lambda(A)\otimes A^k\otimes x(0).$$
For any $A\in\mathbb{R}^{n\times n}$, let ${\rm attr}(A)$ be the set of all starting vectors $x$ from which the sequence $\{A^kx\}_{k=0}^\infty$ reaches an eigenvector of $A$. Namely,
\begin{equation*}
	{\rm attr}(A)=\{x\in\overline{\mathbb{R}}^n: A^kx\in V(A)~\text{for~some}~k\geq0\}.
\end{equation*}
The matrix $A\in\mathbb{R}^{n\times n}$ is called \emph{strongly stable}, if ${\rm attr}(A)=\overline{\mathbb{R}}^n$; called \emph{weakly stable}, if ${\rm attr}(A)=V(A)$~(c.f.\cite{butkovic4}).

In what follows, let us provide some calculations on powers of pseudo-diagonal matrices.
\begin{proposition}\label{power}
Let $D={\rm pdiag}(d_1,\dots,d_n)$ be a pseudo-diagonal matrix with $d_1\le\dots d_s\le0\le d_{s+1}\le\dots\le d_n$, where $n\ge3$ and $1\le s\textless n$. Then for any integer $k\ge2$,
	\begin{equation}\label{4.1}
		(D^k)_{ij}=\begin{cases}
			d_n^{k-2}\quad &i\oplus j\le s,\\
			d^k_i\oplus d^{k-2}_n&i=j\textgreater s,\\
			d^{k-1}_{i\oplus j}\oplus d^{k-2}_n&i\ne j,i\oplus j\textgreater s.
		\end{cases}
	\end{equation}
\end{proposition}
\begin{proof}
Let's prove by induction on $k$.	
For $k=2$, we have that
\begin{equation}\label{dfang}
(D^2)_{ij}=\sum_{t=1}^{n}D_{it}\otimes D_{tj}=\mathop{max}_{t\in [n]}\{D_{it}+ D_{tj}\},\end{equation}
which is equal to $\mathop{max}\{0,2d_i\}$ for $i=j$, and $\mathop{max}\{0,d_i,d_j\}$ for $i\neq j$.	
If $i\oplus j\le s$, then $d_i,d_j\leq0$.  Thus, $(D^2)_{ij}=0$.	
If $i=j\textgreater s$, then $d_i\geq0$. Thus, $(D^2)_{ij}=2d_i=d_i^2=d_i^2\oplus0$.	
If $i\neq j$ and $i\oplus j\textgreater s$, then $d_{i\oplus j}\geq0$ and $d_i,d_j\leq d_{i\oplus j}$. Thus, $(D^2)_{ij}=d_{i\oplus j}=d_{i\oplus j}\oplus0$.

Suppose that the equation (\ref{4.1}) holds for some $k\geq2$, then
	\begin{equation}\label{k+1}
		(D^{k+1})_{ij}=\sum_{t=1}^{n}D_{it}\otimes(D^k)_{tj}=\mathop{max}_{t\in [n]}\{D_{it}+ (D^k)_{tj}\}.
	\end{equation}
	
If $i\oplus j\le s$, then $d_i\leq0$ and $d_i\leq d_n$. By the induction hypothesis, we have that
\begin{equation*}
(D^k)_{tj}=\begin{cases}
d_n^{k-2}\quad &t\leq s,\\
d_t^{k-1}\oplus d^{k-2}_n&t>s.
\end{cases}
\end{equation*}
Hence, $(D^{k+1})_{ij}=\mathop{max}\{d_i+d_n^{k-2},d_{s+1}^{k-1}\oplus d_n^{k-2},\ldots,d_n^{k-1}\oplus d_n^{k-2}\}=d_n^{k-1}$.
	
If $i=j\textgreater s$, by the induction hypothesis, we have that
\begin{equation*}
(D^k)_{ti}=\begin{cases}
d_i^{k-1}\oplus d_n^{k-2}\quad &t\leq s,\\
d^k_i\oplus d^{k-2}_n&i=t>s,\\
d^{k-1}_{i\oplus t}\oplus d^{k-2}_n&i\neq t>s.
\end{cases}
\end{equation*}
Hence, $(D^{k+1})_{ij}=\mathop{max}\{d_i^{k-1}\oplus d_n^{k-2},d_i+(d^k_i\oplus d^{k-2}_n),d^{k-1}_{n}\oplus d^{k-2}_n\}=d_i^{k+1}\oplus d_n^{k-1}$.

If $i\neq j$ and $i\oplus j\textgreater s$, let us discuss the following two cases:	
For $j\textgreater s$, by the induction hypothesis, we have that
\begin{equation*}
		(D^k)_{tj}=\begin{cases}
		d^{k-1}_j\oplus d^{k-2}_n &t\le s,\\
		d^k_j\oplus d^{k-2}_n &j=t\textgreater s,\\
		d^{k-1}_{t\oplus j}\oplus d^{k-2}_n &j\neq t\textgreater s.
		\end{cases}
	\end{equation*}
	Hence, we obtain that
\begin{flalign*}(D^{k+1})_{ij}&=\mathop{max}\{d^{k-1}_j\oplus d^{k-2}_n,d^k_j\oplus d^{k-2}_n,d_i+(d^{k-1}_{i\oplus j}\oplus d^{k-2}_n),d^{k-1}_n\oplus d^{k-2}_n\}\\&=d^k_j\oplus d^{k-1}_n.\end{flalign*}	
For $j\le s,$ we have that $i\textgreater s.$ By the induction hypothesis, we obtain that  \begin{equation*}
		(D^k)_{tj}=\begin{cases}
		 d^{k-2}_n &t\le s,\\
			d^{k-1}_t\oplus d^{k-2}_n &t\textgreater s.
		\end{cases}
	\end{equation*}
Hence, $(D^{k+1})_{ij}=max\{d^{k-2}_n,d^{k-1}_n\oplus d^{k-2}_n,d_i+(d^{k-1}_i\oplus d^{k-2}_n)\}=d^k_i\oplus d^{k-1}_n.$
Therefore, by induction, we complete the proof.
\end{proof}

\begin{remark}
For a general pseudo-diagonal matrix $D={\rm pdiag}(d_1,\dots,d_n)$ whose diagonal entries may not be arranged in ascending order, assume that $n\geq3$ and $d=max\{d_1,\dots,d_n\}\geq 0$, by using similar arguments for proving Proposition \ref{power}, we obtain \begin{equation*}
		(D^k)_{ij}=\begin{cases}
			d^{k-2}\quad &d_i\oplus d_j\le 0,\\
			d^k_{i}\oplus d^{k-2}&i=j,d_i\textgreater 0,\\
			d^{k-1}_{i}\oplus d^{k-1}_{j}\oplus d^{k-2}&i\ne j,d_i\oplus d_j\textgreater 0.
		\end{cases}
	\end{equation*}
For convenience, in what follows we will always assume that the diagonal entries of pseudo-diagonal matrices are arranged in ascending order.
\end{remark}

\begin{corollary}\label{positivedn}
Let $D={\rm pdiag}(d_1,\dots,d_n)$ be a pseudo-diagonal matrix with $0\le d_1\le\dots\le d_n$ and $n\ge3$. Then for any integer $k\ge2$, we have that
	\begin{equation*}
		(D^k)_{ij}=\begin{cases}
			d^k_i\oplus d^{k-2}_n &i=j,\\
			d^{k-1}_{i\oplus j}\oplus d^{k-2}_n&i\ne j.
		\end{cases}
	\end{equation*}
\end{corollary}
\begin{proof}
It is obtained by taking $s=0$ in the proof of Proposition \ref{power}.
\end{proof}
\begin{corollary}\label{negativedn}
Let $D={\rm pdiag}(d_1,\dots,d_n)$ be a pseudo-diagonal matrix with $d_1\le\dots\le d_n\le0$ and $n\ge3$. Then $D^k=0$ for any integer $k\ge2.$
\end{corollary}
\begin{proof}
By \eqref{dfang}, we obtain $D^2=0$. Then using \eqref{k+1}, we get $D^k=0$ for any $k\geq2$.
\end{proof}

Finally, for completeness, we provide calculations on powers of pseudo-diagonal
matrices of $2\times2$ type.
\begin{proposition}\label{2*2}
	Let $D={\rm pdiag}(d_1,d_2)$ be a pseudo-diagonal matrix with $d_1\le d_2$ and let $k\ge2$.

$(1)$~If $d_1\le d_2\le0$, then
$$(D^k)_{ij}=\begin{cases}
				0 &i+j+k\equiv0~({\rm mod}~2),\\
				d_2 &i+j+k\equiv1~({\rm mod}~2).
			\end{cases}$$

$(2)$~If $d_1\le0\le d_2$, then
$$(D^k)_{ij}=\begin{cases}
d_2^{k-2}&i=j=1,\\
d_2^{k-1}&i\neq j,\\
d_2^k&i=j=2.
\end{cases}$$

$(3)$~If $0\le d_1\le d_2$, then
$$(D^k)_{ij}=\begin{cases}
				d_i^k\oplus d_2^{k-2}&i=j,\\
				d_2^{k-1}&i\neq j.
			\end{cases}$$
\end{proposition}
\begin{proof}
When $d_1\le d_2\le 0$, it is easy to see that
	\begin{equation*}
D^{2t}=\begin{pmatrix}
			0&d_2\\
			d_2&0
		\end{pmatrix}
\quad\text{and}\quad
		D^{2t+1}=\begin{pmatrix}
			d_2 &0\\
			0&d_2
		\end{pmatrix}
	\end{equation*}
for any integer $t\ge1.$ Then we obtain the proof of $(1)$.
The proof of the remaining cases is similar to the proof of Proposition \ref{power}.
\end{proof}
\begin{corollary}\label{corollary 4.9}
Let $D={\rm pdiag}(d_1,\dots,d_n)$ be a pseudo-diagonal matrix with $d_1\le\dots\le d_n$ and $n\ge2$. If $n=2$ and $d_2\textless0$, then ${\rm attr}(D)=\{x\in\overline{\mathbb{R}}^2:x_1=x_2\}$ and $D$ is not strongly stable. Otherwise, $D$ is strongly stable.
\end{corollary}
\begin{proof}
If $n=2$ and $d_2\textless0$, let $x\in {\rm attr}(D)$, note that $\lambda(D)=0$, we have that
\begin{equation*}
		\begin{pmatrix}
			d_2 &0\\
			0&d_2
		\end{pmatrix}\otimes \begin{pmatrix}
			x_1\\
			x_2
		\end{pmatrix}=\begin{pmatrix}
			0 &d_2\\
			d_2 &0
		\end{pmatrix}\otimes\begin{pmatrix}
			x_1\\
			x_2
		\end{pmatrix}.
	\end{equation*}
It follows that $x_1=x_2.$ Thus, ${\rm attr}(D)=\{x\in\overline{\mathbb{R}}^2:x_1=x_2\}\subsetneq\overline{\mathbb{R}}^2.$
	
For the other cases, by Lemma \ref{corollary cyclicity} and the Cyclicity Theorem, we obtain that the period of $D$ is $1$. Then there is a positive integer $T$ such that $D^{k+1}=\lambda(D)\otimes D^k$ for any $k\ge T.$ So, $D^{T+1}\otimes x=\lambda(D)\otimes D^T\otimes x$ for any $x\in\overline{\mathbb{R}}^n.$ Hence, $D$ is strongly stable.
\end{proof}

Using the above conclusions, we can answer a question on matrix roots, which was raised in \cite{daniel}. Actually, the matrix power and matrix root problems for $2\times2$ matrices have been completely solved in \cite{daniel}. Let us give the following characterization on matrix roots for any $n\times n$ matrix.
\begin{theorem}{\textup{(\hspace{-0.005cm}\cite{daniel})}}
	Let $A=(a_{ij})\in\mathbb{R}^{n\times n}$ such that
	\begin{equation}\label{4.4}
		a_{ij}a_{tt}\ge a_{it}a_{tj}
	\end{equation}
for any $i,j,t\in[n]$.
Then for each positive integer $k$,
\begin{equation*}
B=(a_{ij}(a_{ii}\oplus a_{jj})^{\frac{1-k}{k}})\in\mathbb{R}^{n\times n}
\end{equation*}
is a $k$ root of $A$, that is, $B^k=A$.
\end{theorem}
\textbf{Question} in \cite{daniel}: Does there exist a matrix $A\in\mathbb{R}^{n\times n}$ and a natural number $k\ge2$ for which $A$ does not satisfy the root condition (\ref{4.4}) but there exists a matrix $B\in\mathbb{R}^{n\times n}$ such that
$B^k=A$? Now we provide an answer as follows.

\begin{proposition}
Given integers $n\ge 3$ and $k\ge 2$. Let $A={\rm pdiag}(a_1,a_2,\dots,a_n)$ be a pseudo-diagonal matrix with $0\le a_1\le\dots\le a_n$. If there exists $2\le i\le n-1$ such that
	\begin{equation*}
		a_i\textgreater (a_n^{k-2}\oplus a_1^{k})^{\frac{1}{k-1}},
	\end{equation*}
	then $D=A^k$ does not satisfy the root condition.
\end{proposition}
\begin{proof}
By Corollary \ref{positivedn}, we have that
	\begin{flalign*}
		D_{11}=a_1^k\oplus a_n^{k-2},D_{1n}=a_n^{k-1},D_{i1}=a_i^{k-1}\oplus a_n^{k-2},D_{in}=a_n^{k-1}.
	\end{flalign*}
	Thus, $D_{11}D_{in}(D_{i1}D_{1n})^{-1}
	=(a^k_1\oplus a_n^{k-2})\otimes (a_i^{k-1})^{-1}
	\textless0.$
	That is, $D$ does not satisfy the root condition (\ref{4.4}).
\end{proof}

\subsection{Optimal-node matrices and separable matrices}
First of all, let us recall the definitions of optimal-node matrices and separable matrices in the max-plus algebra.
\begin{definition}{\textup{(\hspace{-0.005cm}\cite{wanghuili})}}
	Let $A=(a_{ij})\in\mathbb{R}^{n\times n}$. If there exists an integer $k\in[n]$ such that
	\begin{equation*}
		a_{ik}a_{kj}\ge a_{il}a_{lj}
	\end{equation*}
	for any $i,j,l\in [n]$, then $A$ is called an \emph{optimal-node matrix} and every such $k$ is called an \emph{optimal node} of $A$. The set of optimal nodes of $A$ is denoted by $K(A)$.
\end{definition}
\begin{definition}{\textup{(\hspace{-0.005cm}\cite{wanghuili})}}
	Let $A=(a_{ij})\in\mathbb{R}^{n\times n}$. If there are $u_1,\dots,u_n,v_1,\dots,v_n$ in $\mathbb{R}$ such that
	\begin{equation*}
		a_{ij}=u_i+v_j
	\end{equation*}
	for any $i,j\in [n],$ then $A$ is called a \emph{separable matrix.}
\end{definition}
The eigenspace of any optimal-node matrix is one-dimensional. Explicitly, we have the following.
\begin{proposition}{\textup{(\hspace{-0.005cm}\cite{wanghuili})}}\label{optimaleigen}
	Let $A=(a_{ij})$ be an optimal-node matrix and $k\in K(A)$, then $\lambda(A)=a_{kk}$ and the $k$-th column of $A$ spans the eigenspace of $A$, and thus $d(A)=1.$
\end{proposition}

The following two propositions show the invariance of optimal-node matrices and separable matrices under similarity. For simplicity of notation, given a generalized permutation matrix $P\in\overline{\mathbb{R}}^{n\times n}$ with permutation $\pi$, we set $p_i=P_{i,\pi(i)}$ for any $i\in[n]$.
\begin{proposition}\label{sim optimal}
	Let $P\in\overline{\mathbb{R}}^{n\times n}$ be a generalized permutation matrix with permutation $\pi$. Then, $A=(a_{ij})\in\mathbb{R}^{n\times n}$ is optimal-node with $K(A)=\{k_1,\dots,k_m\}$ if and only if $P^{-1}AP$ is optimal-node with $K(P^{-1}AP)=\{\pi(k_1),\dots,\pi(k_m)\}$.
\end{proposition}
\begin{proof}
	If $A$ is optimal-node with $K(A)=\{k_1,\dots,k_m\}$, then we have that
	\begin{equation*}
		a_{ik_t}a_{k_tj}\ge a_{il}a_{lj}
	\end{equation*}
	for any $i,j,l\in [n]$ and $t\in [m].$
	Set $\sigma=\pi^{-1}$ and $b_{ij}=(P^{-1}AP)_{ij}$. By \eqref{xsyg}, we obtain that $b_{ij}=p_{\sigma(i)}^{-1}a_{\sigma(i)\sigma(j)}p_{\sigma(j)}$, and then
	\begin{flalign*}		b_{i,\pi(k_t)}b_{\pi(k_t),j}&=p_{\sigma(i)}^{-1}a_{\sigma(i),\sigma(\pi(k_t))}p_{\sigma(\pi(k_t))}p_{\sigma(\pi(k_t))}^{-1}a_{\sigma(\pi(k_t)),\sigma(j)}p_{\sigma(j)}\\
		&=p_{\sigma(i)}^{-1}a_{\sigma(i),k_t}a_{k_t,\sigma(j)}p_{\sigma(j)}\\
		&\ge p_{\sigma(i)}^{-1}a_{\sigma(i),\sigma(l)}a_{\sigma(l),\sigma(j)}p_{\sigma(j)}\\
		&=p_{\sigma(i)}^{-1}a_{\sigma(i),\sigma(l)}p_{\sigma(l)}p_{\sigma(l)}^{-1}a_{\sigma(l),\sigma(j)}p_{\sigma(j)}\\
		&=b_{il}b_{lj}
	\end{flalign*}
	for any $i,j,l\in [n]$ and $t\in [m]$. Hence, $P^{-1}AP$ is optimal-node and $\pi(k_t)\in K(P^{-1}AP)$, i.e., $\pi(K(A))\subseteq K(P^{-1}AP)$.

Conversely, if $P^{-1}AP$ is optimal-node with $K(P^{-1}AP)=\{\pi(k_1),\dots,\pi(k_m)\}$, by the similar arguments as above, we obtain that $A$ is optimal-node and
$$\pi^{-1}(K(P^{-1}AP))\subseteq K(A).$$
Therefore, we complete the proof.
\end{proof}

\begin{proposition}\label{sim separa}
	Let $P\in\overline{\mathbb{R}}^{n\times n}$ be a generalized permutation matrix with permutation $\pi$. Then, $A\in\mathbb{R}^{n\times n}$ is separable if and only if $PAP^{-1}$ is separable.
\end{proposition}
\begin{proof}
Suppose that $A=(a_{ij})$ is separable, i.e., there are $u_1,\dots,u_n,v_1,\dots,v_n$ such that $a_{ij}=u_iv_j$ for any $i,j\in [n].$
Then we have that
$$(PAP^{-1})_{ij}=p_ia_{\pi(i)\pi(j)}p_{j}^{-1}=p_iu_{\pi(i)}v_{\pi(j)}p^{-1}_j.$$
Set $u'_i=p_iu_{\pi(i)},v'_j=v_{\pi(j)}p_j^{-1}$, then $(PAP^{-1})_{ij}=u'_iv'_j$ for any $i,j\in [n].$ Hence, we finish the proof of necessity, and the proof of sufficiency is similar.
\end{proof}
\begin{proposition}
		Each separable matrix is similar to a symmetric matrix.
	\end{proposition}
	\begin{proof}
	Let $A=(a_{ij})$ be a separable matrix, i.e., there are $u_1,\dots,u_n,v_1,\dots,v_n$ such that $a_{ij}=u_iv_j$ for any $i,j\in [n].$ Let $P={\rm diag}(p_1,\dots,p_n)$ with $p_i=u_i^{-\frac{1}{2}}v_i^{\frac{1}{2}}$ for any $i\in[n]$. Clearly, $P$ is invertible, and $$(PAP^{-1})_{ij}=p_ia_{ij}p_j^{-1}=u_i^{-\frac{1}{2}}v_i^{\frac{1}{2}}u_iv_ju_j^{\frac{1}{2}}v_j^{-\frac{1}{2}}=(u_iv_iu_jv_j)^{\frac{1}{2}}.$$ Hence, $(PAP^{-1})_{ij}=(PAP^{-1})_{ji}$, i.e., $PAP^{-1}$ is symmetric.
	\end{proof}

In what follows, we shall determine which pseudo-diagonalizable matrices are optimal-node or separable. First of all, let us
consider the case of pseudo-diagonal matrices.
	\begin{proposition}\label{4.14}
		Let $A=(a_{ij})\in\mathbb{R}^{n\times n}$ be a pseudo-diagonal matrix with $n\ge2$.
		
		$(1)$ If $n=2$, then $A$ is separable if and only if $a_{11}+a_{22}=0.$
		
		$(2)$ If $n\ge3$, then $A$ is separable if and only if $A=0.$
		
		$(3)$ $A$ is optimal-node if and only if there exists an integer $k\in[n]$ such that $a_{kk}\ge0$ and $a_{ii}\le0$ for any $i\in[n]$ with $i\neq k.$
	\end{proposition}
	\begin{proof}
	Suppose that $A$ is separable, then there are $u_1,\dots,u_n,v_1,\dots,v_n$ such that $a_{ij}=u_iv_j$ for any $i,j\in[n].$ Then for any $i,j\in[n]$ with $i\neq j,$ we have that $$a_{ii}=u_iv_i=u_iv_ju_jv_i(u_jv_j)^{-1}=a_{ij}a_{ji}a_{jj}^{-1}=a_{jj}^{-1}.$$

If $n=2$, then $a_{11}+a_{22}=0.$ If $n\ge3$, for any distinct $i,j,k\in[n]$, we have that $a_{ii}=a_{jj}^{-1}=a_{kk}=a_{ii}^{-1}$, thus $a_{ii}=0$ for any $i\in[n].$ Hence, we have proved the necessity of $(1)$ and $(2)$.

Conversely, if $A=0,$ it's obvious that $a_{ij}=0+0$ for any $i,j\in[n]$, so $A$ is separable. If $n=2$ and $a_{11}a_{22}=0$, let $u_1=-v_2=a_{11}$ and $u_2=v_1=0$, then $a_{ij}=u_iv_j$ for any $i,j\in[n].$ So $(1)$ and $(2)$ have been proved.
	
Suppose that $A$ is optimal-node, then there exists an integer $k\in[n]$ such that \begin{equation}\label{equation in 4.19}
		a_{ik}a_{kj}\ge a_{il}a_{lj}
	\end{equation}
	 for any $i,j,l\in[n].$ Let $i=j=l\neq k$ in (\ref{equation in 4.19}), then we have that $0\ge a_{ll}^2,$ it follows that $a_{ll}\le0$ for any $l\in[n]$ with $l\neq k.$ Let $i=j=k$ and $l\neq k$ in (\ref{equation in 4.19}), then we have that $a_{kk}\ge0.$ Thus, we have proved the necessity for $(3)$.

Conversely, suppose that there exists an integer $k\in[n]$ such that $a_{kk}\ge0$ and $a_{ii}\le0$ for any $i\in[n]$ with $i\neq k.$ Then we have that $a_{ik}a_{kj}\ge0\ge a_{il}a_{lj}$ for any $i,j,l\in[n]$ with $l\neq k$. Therefore, we have completed the proof.		
	\end{proof}

\begin{theorem}\label{4.19}
		Let $A=(a_{ij})\in\mathbb{R}^{n\times n}$ be a pseudo-diagonalizable matrix with $n\ge2$.
		
		$(1)$ If $n=2$, then $A$ is separable if and only if $a_{11}+a_{22}=0.$
		
		$(2)$ If $n\ge3$, then $A$ is separable if and only if the diagonal entries of $A$ are all zero.
		
		$(3)$ $A$ is optimal-node if and only if there exists an integer $k\in[n]$ such that $a_{kk}\ge0$ and $a_{ii}\le0$ for any $i\in[n]$ with $i\neq k.$
	\end{theorem}
 \begin{proof}
From the proof of Theorem \ref{theorem3.5}, we see that there exists a diagonal matrix $P={\rm diag}(p_1,\dots,p_n)$ such that $PAP^{-1}=D$ for some pseudo-diagonal matrix  $D={\rm pdiag}(d_1,\dots,d_n)$. Note that $a_{ii}=p_i^{-1}d_ip_i=d_i$ for any $i\in[n]$.

Let $A\in\mathbb{R}^{2\times 2}$, if $A$ is separable, by Proposition \ref{sim separa}, so is $D$. Thus, $d_1+d_2=0,$ and then $a_{11}+a_{22}=0$.
Conversely, if $a_{11}+a_{22}=0,$ then $d_1+d_2=0$. Thus, $D$ is separable. By Proposition \ref{sim separa}, we obtain that $A$ is separable.
The proofs of $(2)$ and $(3)$ are similar.
    \end{proof}

Using Theorem \ref{4.19} together with Proposition \ref{optimaleigen}, we obtain the following.
\begin{corollary}\label{4last}
	Let $A=(a_{ij})\in\mathbb{R}^{n\times n}$ be a pseudo-diagonalizable matrix such that there exists an integer $k\in[n]$ such that $a_{kk}\ge0$ and $a_{ii}\le0$ for any $i\in[n]$ with $i\neq k$. Then $\lambda(A)=a_{kk}$ and
$V(A)=\{\alpha\otimes A_k:\alpha\in\overline{\mathbb{R}}\}.$
\end{corollary}
\section{The eigenproblem for general pseudo-diagonalizable matrices}
The eigenvalues and eigenspaces of a pseudo-diagonalizable matrix which is optimal-node have been characterized in Corollary \ref{4last}. In this section, we provide a characterization on the eigenvalues and eigenspaces of general pseudo-diagonalizable matrices.

\begin{lemma}\label{lemmalemma}
	Let $a,b,c,d\in\mathbb{R}$, then $(a\oplus b)\otimes(c\oplus d)^{-1}\le a\otimes c^{-1}\oplus b\otimes d^{-1}.$
	\end{lemma}
	\begin{proof}
	If $a\oplus b=a$ and $c\oplus d=c$, then $(a\oplus b)\otimes(c\oplus d)^{-1}=a\otimes c^{-1}\le a\otimes c^{-1}\oplus b\otimes d^{-1}.$
	
	If $a\oplus b=b$ and $c\oplus d=d$, the proof is similar to the previous one.
	
	If $a\oplus b=a$ and $c\oplus d=d$, then $(a\oplus b)\otimes(c\oplus d)^{-1}=a\otimes d^{-1}\le a\otimes c^{-1}\le a\otimes c^{-1}\oplus b\otimes d^{-1}.$
	
	If $a\oplus b=b$ and $c\oplus d=c$, the proof is similar to the previous one.
	
	Therefore, the proof is completed.
	\end{proof}
	\begin{lemma}\label{1+2}
		Let $D={\rm pdiag}(d_1,\dots,d_n)$ be a pseudo-diagonal matrix with $d_1\le\dots\le d_n$ and $n\ge2$. Then $\Gamma(D_\lambda)=D_\lambda\oplus D_\lambda^2.$
	\end{lemma}
	\begin{proof}
		For $n=2$, it is just the definition of $\Gamma(D_\lambda)$. So we assume that $n\ge3$.
		
		If $d_n\le0$, by Lemma \ref{4.17}, we have that $\lambda(D)=0$ and $D_\lambda=D$.
		By definition and Corollary \ref{negativedn}, we obtain that
		$$\Gamma(D_\lambda)=\sum_{k=1}^{n}D_\lambda^k=\sum_{k=1}^{n}D^k=D\oplus0=D\oplus D^2.$$
		
		If $d_1\le\dots\le d_s\le0\le d_{s+1}\le\dots\le d_n$ for some $0\le s\textless n$, by Lemma \ref{4.17}, we have that $\lambda(D)=d_n$ and $D_\lambda=d_n^{-1}\otimes D$. Using Proposition \ref{power}, Corollary \ref{positivedn} and Lemma \ref{lemmalemma}, we obtain that $(D^{k+1})_{ij}\otimes ((D^k)_{ij})^{-1}\le d_n$ for any $i,j\in[n]$ and $k\ge2$. Hence, for any $i,j\in[n]$, we have that
\begin{equation*}
			(D_\lambda^{k+1})_{ij}\otimes ((D_\lambda^k)_{ij})^{-1}=d_n^{-k-1}\otimes (D^{k+1})_{ij}\otimes d_n^{k}\otimes ((D^k)_{ij})^{-1}\le0.\end{equation*}
It follows that $D_\lambda^{n}\leq D_\lambda^{n-1}\leq\cdots \leq D_\lambda^2$. Hence, $D_\lambda^2\oplus\cdots\oplus D_\lambda^{n}=D_\lambda^2$. Therefore, we obtain
that $\Gamma(D_\lambda)=D_\lambda\oplus D_\lambda^2.$
	\end{proof}

\begin{proposition}\label{theorem00000}
Let $D={\rm pdiag}(d_1,\dots,d_n)$  be a pseudo-diagonal matrix with $d_1\le\dots\le d_n\le0$ and $n\ge2.$ Then $\lambda(D)=0,d(D)=1$, $\Gamma(D_{\lambda})=\Gamma(D)=0$ and $$V(D)=\{\alpha\otimes0:\alpha\in\overline{\mathbb{R}}\}.$$
\end{proposition}
\begin{proof}
By Lemma \ref{4.17}, we have that $\lambda(D)=0$ and $D_\lambda=D$. By Lemma \ref{1+2}, we obtain that $\Gamma(D_{\lambda})=\Gamma(D)=D\oplus D^2$. For $n=2$, using Proposition \ref{2*2}, we obtain that \begin{flalign*}
		\Gamma(D)=\begin{pmatrix}
			d_1 &0\\
			0&d_2
		\end{pmatrix}\oplus\begin{pmatrix}
			0&d_2\\
			d_2&0
		\end{pmatrix}=0.
	\end{flalign*}
	For $n\ge3$, using Corollary \ref{negativedn}, we have that $\Gamma(D)=D\oplus0=0.$ The remaining proof is straightforward by Theorem \ref{eigenspace2.4}.
\end{proof}
\begin{corollary}
	Let $D={\rm pdiag}(d_1,d_2)$ be a pseudo-diagonal matrix with $d_1\le d_2\textless0,$ then $D$ is weakly stable, i.e., ${\rm attr}(D)=V(D)$.
\end{corollary}
\begin{proof}
By Corollary \ref{corollary 4.9} and Proposition \ref{theorem00000}, we have that
\begin{equation*}
		{\rm attr}(D)=\{x\in\overline{\mathbb{R}}^2:x_1=x_2\}=V(D).
\end{equation*}	
This finishes the proof.
\end{proof}
\begin{proposition}\label{prop4.26}
	Let $D={\rm pdiag}(d_1,\dots,d_n)$ be a pseudo-diagonal matrix with $d_1\le\dots\le d_s\le0\le d_{s+1}\le\dots\le d_n,n\ge2,0\le s\textless n$ and $d_n\textgreater0.$
Let $k\in[n]$ be the smallest integer such that $d_k=d_n$.
Then $\lambda(D)=d_n,\Gamma(D_\lambda)=D_\lambda,$ $d(D)=n-k+1$ and
	\begin{equation*}
		V(D)=\{\sum_{j=k}^{n}\alpha_j\otimes D_j:\alpha_j\in\overline{\mathbb{R}}, k\leq j\leq n\}.
	\end{equation*}
\end{proposition}
\begin{proof}
	By Lemma \ref{4.17}, we have that $\lambda(D)=d_n$ and $N^*_c(D)=N_c(D)=\{k,k+1\dots,n\}$. By Lemma \ref{1+2}, $\Gamma(D_\lambda)=D_\lambda\oplus D_\lambda^2$. Using Propositions \ref{power}, \ref{2*2} and Corollary \ref{positivedn}, we obtain \begin{flalign*}
		(D_\lambda\oplus D_\lambda^2)_{ij}&=\lambda(D)^{-1}\otimes D_{ij}\oplus\lambda(D)^{-2}\otimes (D^2)_{ij}\\
		&=\begin{cases}
			d_n^{-1}\otimes d_i\oplus d_n^{-2}\otimes0 &i=j\le s,\\
			d_n^{-1}\otimes0\oplus d_n^{-2}\otimes0 &i\oplus j\le s,i\ne j,\\
			d^{-1}_n\otimes d_i\oplus d^{-2}_n\otimes d_i^2&i=j\textgreater s,\\
			d^{-1}_n\otimes0\oplus d^{-2}_n\otimes d_{i\oplus j}&i\ne j,i\oplus j\textgreater s\\
		\end{cases}\\
		&=\begin{cases}
			d_n^{-1}\otimes d_i &i=j,\\
			d_n^{-1} &i\neq j
		\end{cases}\\
		&=(D_\lambda)_{ij}.
	\end{flalign*}
	
Hence, by Theorem \ref{eigenspace2.4}, we obtain that $d(D)=n-k+1$ and \begin{equation*}
		V(D)=\{\sum_{j=k}^{n}\alpha'_j\otimes (D_\lambda)_j:\alpha'_j\in\overline{\mathbb{R}}\}=\{\sum_{j=k}^{n}\alpha_j\otimes D_j:\alpha_j\in\overline{\mathbb{R}}\}.
	\end{equation*}
Therefore, we complete the proof.
\end{proof}

In the end, let us extend the conclusions above to pseudo-diagonalizable matrices.
		\begin{lemma}\label{lemma4.27}
		Let $A=(a_{ij})\in\mathbb{R}^{n\times n}$ be a pseudo-diagonalizable matrix with $n\ge2$ and $a=max\{a_{11},\dots,a_{nn}\}$. Then we have the following:
		
		$(1)$ $\lambda(A)=a\oplus0$.
		
		$(2)$ If $a\textgreater0$, then $i$ is a critical node of $A\Longleftrightarrow a_{ii}=a.$
		
		$(3)$ If $a\textgreater0$, then $N^*_c(A)=N_c(A).$
	\end{lemma}
	\begin{proof}
	For any $\sigma=(i_1,i_2,\dots,i_k)\in\mathcal{P}_n$ with $k\ge2$, clearly, all $i_j$ are distinct. Then by Theorem \ref{theorem3.5}, we obtain that $\omega(\sigma,A)=a_{i_1,i_2}\dots a_{i_{k-1},i_k}a_{i_k,i_1}=0$. The remaining proof is similar to Lemma \ref{4.17}.
	\end{proof}
\begin{theorem}\label{theorem4.28}
	Let $A=(a_{ij})\in\mathbb{R}^{n\times n}$ be a pseudo-diagonalizable matrix with $n\ge2$ and $a=max\{a_{11},\dots,a_{nn}\}\textgreater0$. Let $M(A)$ be the set of integers $i\in[n]$ such that $a_{ii}=a$. Then $\lambda(A)=a,\Gamma(A_\lambda)=A_\lambda$, $d(A)=|M(A)|$ and
	\begin{equation}\label{final}
		V(A)=\{\sum_{m\in M(A)}\alpha_m\otimes A_m:\alpha_m\in\overline{\mathbb{R}}, m\in M(A)\}.
	\end{equation}
\end{theorem}
\begin{proof}
	By Lemma \ref{lemma4.27}, we have that $\lambda(A)=a$ and $N^*_c(A)=N_c(A)=M(A)$, thus $d(A)=|M(A)|.$
	Since $A$ is pseudo-diagonalizable, there exists a pseudo-diagonal matrix $D'={\rm pdiag}(d'_1,\dots,d'_n)$ and invertible matrix $P$, such that $PAP^{-1}=D'$.

Assume $d'_{k_1}\le\dots\le d'_{k_n}$ for some permutation $(k_1,k_2,\ldots,k_n)$ of $[n]$, let $L$ be a permutation matrix with permutation $\pi$, which maps each $i\in[n]$ to $k_i.$ Then $D=LD'L^{-1}$. By (2.4), we obtain that $D_{ij}=D'_{\pi(i),\pi(j)}=\begin{cases}
	d'_{k_i} &i=j\\0 &i\neq j\\
	\end{cases}.$
 Hence, we obtain a pseudo-diagonal form $D$ for $A$ with the diagonal entries arranged in order. By Lemma \ref{similar eigen}, $\lambda(D)=a$. Set $Q=LP.$ Using Proposition \ref{prop4.26}, we obtain that
	\begin{flalign*}&\Gamma(A_\lambda)=\sum_{k=1}^{n}a^{-k}A^k
		=\sum_{k=1}^{n}a^{-k}(Q^{-1}DQ)^k
		=Q^{-1}(\sum_{k=1}^{n}a^{-k}D^k)Q\\&
		=Q^{-1}\Gamma(D_\lambda)Q
		=Q^{-1}D_\lambda Q
		=A_\lambda.\end{flalign*}
	Thus, we get (\ref{final}) by Theorem \ref{eigenspace2.4}. Therefore, we complete the proof.
\end{proof}
	\begin{theorem}
		Let $A=(a_{ij})\in\mathbb{R}^{n\times n}$ be a pseudo-diagonalizable matrix with $n\ge2$ and $a=max\{a_{11},\dots,a_{nn}\}\le0$. Then $\lambda(A)=0$, $\Gamma(A_\lambda)=I\oplus A$, $d(A)=1.$ For finding a basis of the eigenspace of $A$, we just need to take one of the column vectors of $A$, for instance, $A_i$, and replace its $i$-th entry with $0$.
	\end{theorem}
	\begin{proof}
		By Lemma \ref{lemma4.27}, we have that $\lambda(A)=0$. Since $\mu(\sigma,A)=0$ for any $\sigma\in\mathcal{P}_n$ with $l(\sigma)\ge2$, all the critical nodes of $A$ are equivalent. Thus $d(A)=1.$
		By a similar argument as given in the proof of Theorem \ref{theorem4.28}, there exists a pseudo-diagonal matrix $D={\rm pdiag}(d_1,\dots,d_n)$ with $d_1\le\dots\le d_n$ and invertible matrix $P$, such that $P^{-1}AP=D$. Denote the permutation of $P$ by $\pi$. By \eqref{xsyg},
		$d_i=a_{\pi^{-1}(i),\pi^{-1}(i)}\leq0$ for any $i\in[n]$. By Proposition \ref{theorem00000}, we have that $\Gamma(D_\lambda)=0$. Set $p_i=P_{i,\pi(i)}$ for each $i\in[n]$, we obtain that $\Gamma(A_\lambda)=P\Gamma(D_\lambda)P^{-1}=P0P^{-1}$, and then $$\Gamma(A_\lambda)_{ij}=p_i0p_j^{-1}=\begin{cases}
			A_{ij} &i\neq j,\\ 0 &i=j.
		\end{cases}$$ That is, $\Gamma(A_\lambda)=I\oplus A$. Therefore, by Theorem \ref{eigenspace2.4}, we complete the proof.
	\end{proof}

	\section*{Acknowledgments}
	 The work is partially supported by the National Natural Science Foundation of China (No. 12271257) and the Natural Science Foundation of Jiangsu Province of China (No. BK20240137).
%

\end{document}